\documentclass[12pt,a4paper]{article} 
\usepackage[margin=2.5cm]{geometry}
\usepackage{amsmath,amsfonts,amsthm,amssymb}
\usepackage{cleveref}
\usepackage{esint}
\usepackage{color}
\usepackage{bbm}
\usepackage{url}

\newcommand{\bN}{\mathbb{N}}
\newcommand{\bR}{\mathbb{R}}
\newcommand{\bT}{\mathbb{T}}
\newcommand{\cA}{\mathcal{A}}
\newcommand{\cF}{\mathcal{F}}

\newcommand{\cJ}{\mathcal{J}}
\newcommand{\cM}{\mathcal{M}}
\newcommand{\cN}{\mathcal{N}}
\newcommand{\cY}{\mathcal{Y}}
\newcommand{\rd}{\mathrm{d}}
\newcommand{\Aut}{\mathrm{Aut}}
\newcommand{\cha}{\mathbbm{1}}
\newcommand{\norm}[1]{\left\|{#1}\right\|}

\theoremstyle{plain}%
\newtheorem{theorem}{Theorem}%
\newtheorem{proposition}{Proposition}%
\newtheorem{lemma}{Lemma}%
\newtheorem{corollary}{Corollary}%
\newtheorem{conjecture}{Conjecture}%

\theoremstyle{definition}%
\newtheorem{definition}{Definition}%

\theoremstyle{remark}%
\newtheorem{remark}{Remark}%

\title{Mixing for generic passive scalars by incompressible flows}

\author{Zeyu Jin \footnote{School of Mathematical Sciences, Peking University, Beijing 100871, China (\texttt{jinzy@pku.edu.cn}).} 
\and Ruo Li
\footnote{CAPT, LMAM and School of Mathematical Sciences, Peking University, Beijing 100871, China; Chongqing Research Institute of Big Data, Peking University, Chongqing 401121, China
(\texttt{rli@math.pku.edu.cn}).}}

\date{}

\begin{document}

\maketitle

\begin{abstract}
Mixing by incompressible flows is a ubiquitous yet incompletely 
understood phenomenon in fluid dynamics. 
While previous studies have focused on optimal mixing rates, the question of 
its genericity, i.e., whether mixing occurs for typical incompressible flows 
and typical initial data, remains mathematically unclear.
In this paper, it is shown that classical mixing criteria, 
e.g. topological mixing or non-precompactness in $L^2$ for all nontrivial 
densities, 
fail to persist under arbitrarily small perturbations of velocity fields.
A Young-measure theory adapted to $L^\infty$ data is then developed to 
characterize exactly which passive scalars mix.
As a consequence, the existence of a single mixed density is equivalent to 
mixing for generic bounded data, and this equivalence is further tied to the 
non-precompactness of the associated measure-preserving flow maps in $L^p$.
These results provide a foundation for a general theory of generic mixing in 
non-autonomous incompressible flows.
\newline

\newline

\end{abstract}

\section{Introduction}
Mixing by incompressible flows is a fundamental phenomenon in fluid 
dynamics, arising across diverse scales and applications ranging 
from industrial processes, such as
chemical reactions \cite{BALDYGA1997457,Koochesfahani_Dimotakis_1986} and  
food processing \cite{cullen2009food,nienow1997mixing}, 
to geophysical systems, including 
atmospheric transport \cite{haynes2005transport} and 
ocean circulation \cite{garrett1979mixing}.
In these settings, the stirring and blending of the fluid play a critical 
role in energy transfer and the generation of turbulence 
\cite{cotizelatiMixingIncompressibleFlows2024}.
Despite its ubiquity, a complete mathematical characterization of mixing 
remains elusive. 
Although extensive work has addressed mixing in various contexts, such as 
statistical properties \cite{boffettaNonasymptoticPropertiesTransport2000}, 
inviscid damping \cite{Bedrossian2016enhanced}, 
chaotic dynamics 
\cite{ottinoKinematicsMixingStretching2004, Hassan1984stirring}, and 
optimal control theory 
\cite{liuMixingEnhancementOptimal2008, mathewOptimalControlMixing2007}, 
a unified understanding of when and how mixing occurs under general 
incompressible flows is still lacking.

Mathematically, the evolution of a passive scalar density $\rho(t,x)$ 
transported by a divergence-free velocity field $u(t,x)$ provides a 
fundamental model for the mixing phenomenon.
It is governed by the transport equation:
\[
\partial_t \rho + u \cdot \nabla \rho = 0, \quad 
\rho(0, \cdot) = \rho_0.
\]
Here $\rho$ is passive, meaning that its evolution does not influence the 
velocity field $u$, and diffusion is neglected, isolating the role of 
advection in mixing.
Since the transport equation conserves all the $L^p$ norms of $\rho$, 
one employs weaker notions to quantify the mixing rates.
Two common approaches are 
\emph{functional mixing scales} \cite{Lin2011optimal,mathew2005multiscale}, 
which quantify the decay of negative-index Sobolev norms of $\rho$, and 
\emph{geometric mixing scales} \cite{bressanLemmaConjectureCost2003}, 
which measure how rapidly the scalar becomes spatially homogenized on finer 
and finer scales.
Qualitatively, mixing corresponds to the weak convergence in $L^2$ of 
$\rho(t, \cdot)$ to its spatial average as $t \to \infty$ in either framework 
\cite{mathew2005multiscale}.

A central theme in the literature is the determination of optimal mixing 
rates under various regularity and integrability constraints on $u$,
motivated by applications in industrial processes and control theory, 
where mixing efficiency is critical 
\cite{albertiExponentialSelfsimilarMixing2018a}.
When $u$ is uniformly bounded in $W^{s,p}$ with $s \in [0,1)$ and 
$p \in [1,\infty]$, it is known that perfect mixing, i.e., convergence to its 
spatial average, can occur in finite time 
\cite{Lin2011optimal,Lunasin2012Optimal}, which is a consequence of the 
non-uniqueness phenomena for the transport equation in this regime
\cite{alberti2014uniqueness,mondena2018nonuniqueness}.
At the critical regularity $s = 1$, if $u$ is uniformly bounded in $W^{1,p}$
with $p \in (1, \infty]$, exponential decay of mixing 
has been obtained in several senses, including 
geometric mixing scales \cite{Crippa2008Estimates},
decay of the Monge--Kantorovich--Rubinstein distance and the $H^{-1}$ norm 
for binary-phase initial data \cite{seisMaximalMixingIncompressible2013},
and $H^{-1}$-decay for general initial data \cite{IyerKiselevXu2014}
The borderline case $s = 1, p = 1$ corresponds to Bressan’s 
longstanding mixing conjecture \cite{bressanLemmaConjectureCost2003}, 
which remains unsolved.
On the constructive side, explicit velocity fields achieving these optimal 
mixing rates have been constructed in 
\cite{albertiExponentialSelfsimilarMixing2018a, 
elgindiUniversalMixersAll2019, yaoMixingUnmixingIncompressible2017}.
Moreover, these mixing results have further implications for 
loss of regularity of solutions to the transport equation 
\cite{Alberti2019Loss}, 
enhanced dissipation in the presence of diffusion 
\cite{constantinDiffusionMixingFluid2008, coti2020relation, 
Feng2019dissipation}, and anomalous dissipation in turbulent regimes 
\cite{Elgindi2024norm, drivas2022anomalous}.

While much is known about achieving the fastest possible mixing under 
prescribed constraints, a fundamental question of equal importance, namely, 
\emph{to what extent mixing is a generic property of incompressible flows},
has received relatively little attention.
From a physical standpoint, one expects mixing to be prevalent among 
incompressible flows due to its ubiquity, yet the mathematical theory of 
generic mixing is largely undeveloped. 
To date, genericity results have been obtained only in very specialized 
settings: 
for autonomous flows, it is shown that topological mixing is generic in the 
$C^1$ topology \cite{bessaGenericIncompressibleFlow2008}; 
and for autonomous shear flows, quantitative estimates of mixing rates hold 
for a generic set of initial data in a measure-theoretic sense 
\cite{galeatiMixingGenericRough2023}. 
However, these results exclude the non-autonomous flows that 
are ubiquitous in practical applications. 
In particular, no theory yet explains whether the solution of the transport 
equation is mixed by a typical time-dependent divergence-free velocity field.
A theory of the genericity of mixing for general non-autonomous flows 
still remains unaddressed.

In this work, we initiate a systematic study of generic mixing for 
non-autonomous incompressible flows. 
Theorem~\ref{thm:nong} shows that all classical mixing 
criteria, including topological mixing, geometric mixing, or functional mixing, 
fail to persist under arbitrarily small perturbations of the velocity field in 
$L^\infty ([0, +\infty); W^{1,p})$ with $p \in (1, \infty)$. 
This non-robustness indicates that one must seek an even weaker, yet still 
meaningful, notion of generic mixing. 
To achieve this, finer tools are required to capture long-time behavior of the 
transported density $\rho(t, \cdot)$. 
An appropriate tool is Young measures \cite{Young1942Generalized}.
In particular, we develop in Theorem~\ref{thm:youngL} a version of the 
fundamental theorem of Young measures adapted to $L^\infty$ data. 
Within this framework, Theorem~\ref{thm:cha} provides a complete 
characterization of those initial densities that mix under a given 
incompressible flow by identifying the $\sigma$-algebra generated by unmixed 
level sets.
As a consequence, the existence of a single mixed density is equivalent to 
mixing for generic bounded data. 
Finally, in Theorem~\ref{thm:topo} we show that this mixing criterion for 
generic data is equivalent to the non-precompactness of the family of 
measure-preserving flow maps in the $L^p$ topology with $p \in [1, \infty)$.

We emphasize that the distinction between mixing for \emph{all} nontrivial 
initial data and mixing for a \emph{generic} set of initial data is essential. 
Indeed, this perspective of generic mixing aligns with fundamental questions in 
fluid mechanics such as \v{S}ver\'ak's conjecture 
\cite{sverak2011course,drivas2023singularity} 
on the long-time behavior of two-dimensional incompressible Euler flows, which 
posits that the vorticity orbits are not precompact in $L^2$ for \emph{generic}
bounded initial vorticities. 
To our knowledge, the investigations presented herein constitute the first 
rigorous framework for understanding the genericity of mixing in non-autonomous 
incompressible flows.

The rest of this paper is organized as follows. 
In \Cref{sec:nong}, we review several classical definitions of mixing and 
prove their non-robustness under small 
$L^\infty ([0,+\infty); W^{1,p})$ perturbations. 
In \Cref{sec:young}, we prove a version of the fundamental theorem of Young 
measures for measure-preserving flow maps and derive a Young-measure 
characterization of mixing.
\Cref{sec:structure} employs this characterization to identify the 
$\sigma$-algebra structure of unmixed sets and characterizes the 
structure of mixed initial data.
\Cref{sec:topo} discusses the topology of measure-preserving bijections and 
establishes that non-precompactness of the flow maps is equivalent to mixing 
for a generic set of initial data in $L^\infty$.
Finally, in \Cref{sec:concl}, we make some conclusive remarks and propose a 
conjecture on the genericity of incompressible flows.

\section{Non-robustness of mixing}\label{sec:nong}
We would like to identify a notion of mixing that can be generic in the space 
of incompressible flows. 
To achieve this, we have to investigate weaker formulations of mixing,
since broader definitions admit larger classes of flows and thus have 
greater prospects for genericity.
Throughout, we interpret genericity in the topological sense rather than in 
the measure-theoretic sense.

The objective of this section is to demonstrate that the conventional notions 
of mixing cannot be robust.
In \Cref{sec:nong_notions}, we recall the classical concepts of mixing, 
such as strong mixing, topological mixing, and mixing of transport solutions 
for essentially arbitrary initial data.
In \Cref{sec:nong_thm}, we prove that none of these notions of mixing 
persists under small perturbations of the velocity field in the 
$L^\infty([0,+\infty);W^{1,p})$ topology. 
This non-robustness compels us to seek an even weaker mixing criterion, which 
will be discussed in \Cref{sec:structure}.

\subsection{Setups}\label{sec:nong_setup}
Let $d\ge2$ and consider the $d$-dimensional torus $\bT^d$ equipped with the 
Lebesgue measure~$\mu$.
Consider the following transport equation: 
\begin{equation}\label{eq:transport}
\partial_t \rho + u \cdot \nabla \rho = 0, \quad 
\rho(0, \cdot) = \rho_0,
\end{equation}
where $u : [0, +\infty) \times \bT^d \to \bR^d$ is a divergence-free velocity 
field, i.e., $\nabla \cdot u = 0$ in the sense of distributions, and 
$\rho : [0, +\infty) \times \bT^d \to \bR$ is the transported scalar density.
We work on $\bT^d$ to avoid technical complications introduced by 
boundary issues and non-compactness of the domain, 
though all results can be extended to the full space 
$\bR^d$ or general compact Riemannian manifold.

Fix $p \in (1, \infty)$ and assume that 
\[
u \in L^\infty([0, +\infty); W^{1,p}(\bT^d; \bR^d)), \quad 
\rho_0 \in L^\infty(\bT^d).
\]
By the DiPerna--Lions theory 
\cite{DiPerna1989Ordinary,Ambrosio2004Transport}, 
there is a unique weak solution to \cref{eq:transport}, 
\[\rho \in L^\infty([0,+\infty); L^\infty(\bT^d)) \cap 
C([0,+\infty); L^p(\bT^d)),\]
which is represented via the measure-preserving flow map 
$\Phi_t : \bT^d \to \bT^d$ by 
\[
\rho(t,x) = \rho_0(\Phi_t^{-1}(x)),
\]
where, for almost every $x \in \bT^d$, $\Phi_t(x)$ is the unique 
absolutely continuous integral solution of the ODE 
$\dot{\gamma}(t) = u(t, \gamma(t))$ with $\gamma(0) = x$ 
(cf. \cite{ambrosioContinuityEquationsODE2014,
DiPerna1989Ordinary}).

Denote by $\cM$ the $\sigma$-algebra of all the Lebesgue-measurable sets on 
$\bT^d$, by $\cJ\subset\cM$ the $\sigma$-ideal of null sets, and write
\[
L^\infty (W^{1,p}) = L^\infty([0, +\infty); W^{1,p}(\bT^d; \bR^d))
\]
when there is no ambiguity.

\subsection{Notions of mixing}\label{sec:nong_notions}
We begin by recalling the quantitative estimates for mixing in the $H^{-1}$ 
norm in \cite{IyerKiselevXu2014}.
For any divergence-free velocity field $u \in L^\infty (W^{1,p})$ and any 
$\rho_0 \in L^\infty$, one has 
\[
\norm{\rho(t) - \bar{\rho}}_{H^{-1}} \ge c_1 \exp \left( 
- c_2 \int_0^t \norm{\nabla u(s)}_{L^p} \,\rd s \right),
\]
where $\bar{\rho} := \fint_{\bT^d} \rho_0(x) \,\rd x$, and the constants
$c_1, c_2 > 0$ depend explicitly only on $p$, $d$, and $\rho_0$.
Thus, exponential decay of the $H^{-1}$-norm is optimal, and in fact examples 
achieving this bound have been constructed explicitly in 
\cite{albertiExponentialSelfsimilarMixing2018a}.

In general, one cannot expect exponential decay of 
$\norm{\rho(t) - \bar{\rho}}_{H^{-1}}$ for each divergence-free velocity 
field $u$. 
A natural question is whether 
\[
\norm{\rho(t) - \bar{\rho}}_{H^{-1}} \to 0 \quad \text{ as }
t \to \infty
\]
holds for a \emph{generic} incompressible flow $u \in L^\infty (W^{1,p})$.
Here and below we restrict attention to functional mixing scales, since 
by \cite{mathew2005multiscale}, the decay of functional and geometric 
mixing scales are equivalent.

Since $\norm{\rho(t)}_{L^q} = \norm{\rho_0}_{L^q}$ for all 
$q \in [1, \infty]$, it follows that 
\[
\sup_{t \ge 0} \norm{\rho(t) - \bar{\rho}}_{L^2} < +\infty.
\]
and hence by the Banach--Alaoglu theorem, 
\begin{equation}\label{eq:weakconst}
\norm{\rho(t) - \bar\rho}_{H^{-1}} \to 0
\quad \Longrightarrow \quad
\rho(t) \rightharpoonup \bar\rho \quad \text{ weakly in } L^2.
\end{equation}
The latter condition is equivalent to \emph{strong mixing} in dynamical 
systems:
\[
\mu(\Phi_t(A) \cap B) \to \frac{\mu(A) \mu(B)}{\mu(\bT^d)} \quad 
\text{as } t \to \infty,
\]
for every pair of measurable sets $A, B \in \cM$.

One may weaken the notion of strong mixing further by considering 
\emph{topological mixing}: 
for any two nonempty open sets $U, V \subset \bT^d$, there exists 
$\tau > 0$ such that 
\[
\Phi_t(U) \cap V \not = \varnothing \quad \text{for all } t \ge \tau.
\]
It is known that, in dimension $d\ge3$, topological mixing is a generic 
property in the sense of Baire category among time-independent $C^1$ 
incompressible flows \cite{bessaGenericIncompressibleFlow2008}.

From a purely functional-analytic perspective, one may further relax the 
mixing criterion in \cref{eq:weakconst} to the condition that the trajectory 
${\rho(t, \cdot)}_{t\ge0}$ is \emph{not precompact} in $L^2$. More precisely:

\begin{definition}[Mixing]\label{def:mixing}
Fix a divergence-free velocity field $u \in L^\infty (W^{1,p})$. 
An initial datum $\rho_0 \in L^\infty(\bT^d)$ is said to be 
\emph{mixed} by $u$ if the set $\{\rho(t, \cdot)\}_{t \ge 0}$ is 
not precompact in $L^2$.
Denote by 
\[
\cF_u := \left\{ \rho_0 \in L^\infty(\bT^d) \mid \{\rho(t, \cdot)\}_{t \ge 0} 
\text{ is not precompact in } L^2 \right\}
\]
the set of all such initial data.
\end{definition}

We say that $u$ is \emph{universally non-precompact} if 
\[
\cF_u = L^\infty(\bT^d) \setminus \{ \rho_0 \in L^\infty(\bT^d) \mid 
\rho_0(x) = \bar{\rho} \text{ a.e. } x \in \bT^d \}.
\]
Universal non-precompactness is strictly weaker than 
either $H^{-1}$-norm decay or strong mixing, 
since $\norm{\rho(t) - \bar{\rho}}_{L^2}$ remains bounded below away from 
zero, unless $\rho_0 = \bar\rho$ almost everywhere.

\begin{remark}
Physically, these weaker notions capture the fact that pure advection may 
fail to be so effective to achieve uniform spatial homogenization. 
This phenomenon is also observed in Euler flows 
\cite{drivas2023singularity,Shnirelman1993lattice,sverak2011course}, where 
the vorticity solutions can be trapped in time-dependent regimes that 
perpetually avoid further mixing. 
In such cases, they develop intricate structures rather than homogenization.
\end{remark}

\subsection{Failure of robustness}\label{sec:nong_thm}
In the previous discussions, we examined several natural weakenings of mixing.
We now demonstrate that none of these properties is preserved under 
arbitrarily small perturbations in $L^\infty (W^{1,p})$. 
To be precise, we have the following result.

\begin{theorem}\label{thm:nong}
Let $u \in L^\infty (W^{1,p})$ be a divergence-free velocity field, and assume 
that the family $\{|\nabla u(t, \cdot)|^p\}_{t \ge 0}$ is uniformly 
integrable over $\bT^d$. 
Then for every $\epsilon > 0$, there exists a divergence-free velocity field 
$v \in L^\infty (W^{1,p})$ with 
\begin{equation}\label{eq:nong_pert_bound}
\sup_{t \ge 0} \norm{v(t, \cdot)}_{W^{1, p}} < \epsilon,
\end{equation}
such that $u + v$ is neither topologically mixing nor universally 
non-precompact.
\end{theorem}

\begin{proof}
Since the set of all the smooth divergence-free velocity fields is dense in 
\[
\{ u \in W^{1,p} \mid \nabla \cdot u = 0 \},
\]
we may assume without loss of generality that $u(t, x)$ is smooth in~$x$.
Fix $\delta \in (0, \frac14)$ and choose a point $\bar{x} \in \bT^d$
such that the trajectory $t \mapsto \Phi_t(\bar{x})$ is absolutely 
continuous on $[0, +\infty)$.
Let $\phi_\delta \in C^\infty(\bR^d)$ be a radial cutoff function 
satisfying that 
\[
\phi_\delta(y) = 
\begin{cases}
1, & \quad |y| \le \delta, \\ 0, & \quad |y| \ge 2 \delta,
\end{cases}
\]
and 
\begin{equation}\label{eq:cutoff_deriv}
\norm{\nabla \phi_\delta}_{L^\infty} \lesssim \delta^{-1}.
\end{equation}
For each fixed $t \ge 0$, set $y = x - \Phi_t(\bar{x})$ and consider the 
annular domain 
\[
\Omega_\delta := B_{2\delta}(0) \setminus \overline{B_\delta(0)}.
\]
Define the scalar function 
\[
f_t(y) := \big( u(t, y + \Phi_t(\bar{x})) - u(t, \Phi_t(\bar{x})) \big) \cdot 
\nabla \phi_\delta(y),
\]
which has zero integral over~$\Omega_\delta$. 
Indeed, one checks by integration by parts that
\begin{multline}\label{eq:bogovskii_compat}
\int_{\Omega_\delta} f_t(y) \,\rd y = 
\int_{\Omega_\delta}
\big( u(t, y + \Phi_t(\bar{x})) - u(t, \Phi_t(\bar{x})) 
\big) \cdot \nabla \phi_\delta(y) \,\rd y \\ 
= -\int_{\Omega_\delta} \phi_\delta(y)
\nabla \cdot u(t, y + \Phi_t(\bar{x})) \,\rd y + 
\int_{\partial \Omega_\delta} \phi_\delta(y) 
\big( u(t, y + \Phi_t(\bar{x})) - u(t, \Phi_t(\bar{x})) \big) \cdot n 
\,\rd \sigma(y) \\
= - \int_{\partial B_\delta} u(t, y + \Phi_t(\bar{x})) \cdot n 
\,\rd \sigma(y) = 
- \int_{B_\delta} \nabla \cdot u(t, y + \Phi_t(\bar{x})) \,\rd y = 0,
\end{multline}
since the velocity field $u$ is divergence-free.
By the standard Bogovski\u{\i} theory 
(cf. \cite[Theorem~1 and Lemma~3]{Bogovskii1980solutions} and 
\cite[Theorem~III.3.1]{Galdi2011introduction}), 
there exists a unique vector field 
$w_t \in W^{1,p}_0(\Omega_\delta; \bR^d)$ solving the following 
Bogovski\u{\i} equation:
\begin{equation}\label{eq:Bogovskii}
\nabla \cdot w_t(y) = f_t(y)
\quad \text{in } \Omega_\delta, \quad 
w_t|_{\partial \Omega_\delta} = 0, 
\end{equation}
and satisfying that 
\begin{equation}\label{eq:bogovskii_estimate}
\norm{w_t}_{W^{1,p}(\Omega_\delta; \bR^d)} \lesssim 
\norm{f_t}_{L^p(\Omega_\delta)}.
\end{equation}
Here, \cref{eq:Bogovskii} satisfies the compatibility condition due to 
\cref{eq:bogovskii_compat}.
Moreover, by Poincare inequality, one obtains that 
\[
\norm{u(t, \cdot + \Phi_t(\bar{x})) - u(t, \Phi_t(\bar{x}))}
_{L^p(B_{2 \delta}(0); \bR^d)} \lesssim \delta 
\norm{\nabla u(t, \cdot)}_{L^p(B_{2\delta}(\Phi_t(\bar{x})); \bR^d)}.
\]
By \cref{eq:cutoff_deriv}, one has the following estimate
\begin{equation}\label{eq:poincare}
\norm{f_t}_{L^p(\Omega_\delta)} \lesssim
\norm{\nabla u(t, \cdot)}_{L^p(B_{2\delta}(\Phi_t(\bar{x})); \bR^d)},
\end{equation}
which yields that 
\begin{equation}\label{eq:nong_w}
\norm{w_t}_{W^{1,p}} \lesssim 
\norm{\nabla u(t, \cdot)}_{L^p(B_{2\delta}(\Phi_t(\bar{x})); \bR^d)},
\end{equation}
thanks to \cref{eq:bogovskii_estimate}.
Define, for each fixed $t \ge 0$, the localized perturbation as follows,
\[
v(t, x) := \phi_\delta(x - \Phi_t(\bar{x})) \cdot 
\big( u(t, \Phi_t(\bar{x})) - u(t, x) \big) + w_t(x - \Phi_t(\bar{x})),
\]
extended by zero outside $B_{2\delta}(\Phi_t(\bar x))$.
By construction, $v(t,\cdot)$ is divergence-free and compactly supported in 
$B_{2\delta}(\Phi_t(\bar x))$.
Combining \cref{eq:poincare,eq:nong_w}, we can obtain that 
\[
\norm{v(t, \cdot)}_{W^{1,p}(\bT^d; \bR^d)} \lesssim 
\norm{\nabla u(t, \cdot)}_{L^p(B_{2\delta}(\Phi_t(\bar{x})); \bR^d)}.
\]
Since $\{|\nabla u(t, \cdot)|^p\}_{t \ge 0}$ is uniformly integrable over 
$\bT^d$, by choosing $\delta > 0$ sufficiently small, we can ensure 
\cref{eq:nong_pert_bound}.
Moreover, on the spherical shell $\partial B_\delta(\Phi_t(\bar x))$, 
one has that 
\[
u(t, x) + v(t, x) = u(t, \Phi_t(\bar{x})) \quad \text{for each }
x \in \partial B_\delta(\Phi_t(\bar{x})).
\]
In addition, $v(t, \Phi_t(\bar{x})) \equiv 0$ for all $t \ge 0$.
Hence the perturbed flow map $\tilde{\Phi}_t$ generated by $u + v$ satisfies 
that 
\begin{equation}\label{eq:phi_b}
\tilde{\Phi}_t(B_\delta(\bar{x})) = B_\delta (\tilde{\Phi}_t(\bar{x})) \quad 
\text{for all } t \ge 0.
\end{equation}
For an arbitrary sequence $\{t_j\}_{j \in \bN}$ such that 
$t_j$ tends to infinity, one can choose a subsequence 
$\{t_{j_k}\}_{k \in \bN}$ such that
\[
\tilde{\Phi}_{t_{j_k}}(\bar{x}) \to \tilde{x} \in \bT^d \quad 
\text{ as } k \to \infty,
\]
since $\tilde\Phi_t(\bar x)$ remains in a compact set $\bT^d$.
Choose any open set $V\subset\bT^d$ for which
$V \cap B_{2\delta}(\tilde x)=\varnothing$. 
It follows that 
\[
\tilde{\Phi}_{t_{j_k}}(B_{\delta}(\bar{x})) \cap V = \varnothing,
\]
for sufficiently large $k$, due to \cref{eq:phi_b}.
Therefore, $u + v$ fails to be topologically mixing.
Setting $\rho_0 = \cha_{B_\delta(\bar{x})}$, one sees that 
\[
\rho(t_{j_k}, \cdot) = \cha_{B_\delta(\Phi_{t_{j_k}}(\bar{x}))} \to 
\cha_{B_\delta(\tilde{x})} \quad \text{strongly in } L^2,
\]
showing that $u + v$ is not universally non-precompact.
This completes the proof.
\end{proof}

\begin{remark}
The uniform integrability assumption on the family
$\{ |\nabla u(t, \cdot)|^p \}_{t \ge 0}$ is mild.
In particular, it holds automatically for any time-independent velocity 
field in $W^{1,p}$.
Furthermore, for any $u \in L^\infty (W^{1, p+\epsilon})$ for some 
arbitrarily small $\epsilon > 0$, H\"older’s inequality yields that 
\[
\norm{\nabla u(t, \cdot)}_{L^p(D)} \lesssim 
\norm{u}_{L^\infty (W^{1,p+\epsilon})} \cdot 
\mu(D)^{\frac{\epsilon}{p (p+\epsilon)}} \quad \text{for each measurable set }
D \in \cM,
\]
and hence $\norm{\nabla u}_{L^p(D)}$ can be made arbitrarily small as long as 
$\mu(D)$ is sufficiently small.
In particular, all known constructions achieving optimal exponential mixing 
rates (e.g. 
\cite{albertiExponentialSelfsimilarMixing2018a,elgindiUniversalMixersAll2019})
satisfy this uniform integrability condition.
Consequently, even flows that exhibit the fastest possible mixing rates can 
be perturbed arbitrarily slightly in $L^\infty (W^{1,p})$ to violate every 
standard mixing criterion, let alone preserve the optimal mixing rates.
This result may be relevant to applications requiring rapid mixing, 
where idealized optimal flows are unattainable and small perturbations 
are inevitable.
\end{remark}

\section{Young measures}\label{sec:young}
In this section, we develop a Young–measure framework to describe the 
asymptotic behavior of the transported density
$\rho(t, \cdot) = \rho_0 \circ \Phi_t^{-1}$.
This characterization will be essential in \Cref{sec:structure} to determine 
the structure of mixed initial data.

We review in \Cref{sec:young_ball} the classical fundamental theorem of 
Young measures due to Ball \cite{Ball1989Young}, and then refine this result 
to the case of measure-preserving maps and arbitrary $L^\infty$ data in 
\Cref{sec:young_refine}.
Using these results, we characterize exactly when the trajectory 
$\{\rho \circ \Phi_t^{-1}\}_{t \ge 0}$ is $L^2$-precompact in 
\Cref{sec:young_cha}.

\subsection{Fundamental theorem of Young measures}\label{sec:young_ball}
We begin by recalling the classical theorem of Ball \cite{Ball1989Young} in a 
form adapted to our setting.

\begin{theorem}[Fundamental theorem of Young measures 
\cite{Ball1989Young}]\label{thm:ball}
Let $\{\Psi_j\}_{j \in \bN}$ be a sequence of Lebesgue measurable 
maps $\Psi_j : \bT^d \to \bT^d$. 
Then there exists a subsequence $\{\Psi_{j_k}\}_{k \in \bN}$ 
and a family of probability measures $\{\nu^x\}_{x \in \bT^d}$ on~$\bT^d$, 
such that $\nu^x(\bT^d) = 1$ for almost every $x \in \bT^d$, 
and 
\[
\rho \circ \Psi_{j_k} \rightharpoonup \langle \nu^x, \rho \rangle \quad 
\text{weakly in } L^2(\bT^d),
\]
for any continuous function $\rho \in C(\bT^d)$.
\end{theorem}

\begin{remark}
Theorem~\ref{thm:ball} requires the continuity of $\rho$ and fails for 
arbitrary measurable functions, as the composition of a merely measurable 
function $\rho$ with a measurable map $\Psi_j$ need not be measurable. 
In our problem, however, $\rho_0\in L^\infty(\mathbb{T}^d)$, and thus, 
\Cref{thm:ball} does not apply directly.
Crucially, the maps $\Psi_j$ are measure-preserving bijections, which 
allows us to extend the conclusion to all bounded measurable $\rho$.
\end{remark}

\subsection{Extension to $L^\infty$}\label{sec:young_refine}
\begin{theorem}\label{thm:youngL}
Let $\{\Psi_j\}_{t \in \bN}$ be a sequence of measure-preserving bijections
$\Psi_j : \bT^d \to \bT^d$. 
Let the subsequence $\{\Psi_{j_k}\}_{k \in \bN}$ and the family of probability 
measures $\{\nu^x\}_{x \in \bT^d}$ be as in 
Theorem~\ref{thm:ball}.
Then 
\[
\rho \circ \Psi_{j_k} \rightharpoonup \langle \nu^x, \rho \rangle \quad 
\text{weakly in } L^2(\bT^d),
\]
for every $\rho \in L^\infty(\bT^d)$.
\end{theorem}

\begin{proof}
By \Cref{thm:ball}, the conclusion holds for every $\rho \in C(\bT^d)$.
Hence for each such continuous $\rho$ and each $f \in L^2(\bT^d)$, one has 
that 
\[
\int_{\bT^d} (\rho \circ \Psi_{j_k})(x) f(x) \,\rd x \to
\int_{\bT^d} \langle \nu^x, \rho \rangle f(x) \,\rd x \quad \text{as }
k \to \infty.
\]
For any measurable set $D \in \cM$, taking a continuous 
function $\rho$ satisfying that $0 \le \rho \le \cha_D$, and the test function 
$f \equiv 1$, one has that 
\[
\int_{\bT^d} \rho \,\rd x = \int_{\bT^d} \rho \circ \Phi_{j_k} \,\rd x 
\to \int_{\bT^d} \int_{\bT^d} \rho(y) \,\rd \nu^x(y) \,\rd x \ge 
\int_{\bT^d} \nu^x(D) \,\rd x.
\]
Taking the infimum over all such $\rho$, one can obtain that 
\begin{equation}\label{eq:young_set}
\int_{\bT^d} \nu^x(D) \,\rd x \le \mu(D).
\end{equation}
To extend the conclusion of \Cref{thm:ball} to $L^\infty$, fix 
$\rho \in L^\infty(\bT^d)$ and $\epsilon > 0$.
By Lusin's theorem, there exists a 
compact set $K \subset \bT^d$ with 
\begin{equation}\label{eq:young_lusin}
\mu(\bT^d \setminus K) < \epsilon,
\end{equation}
such that $\rho$ restricted to $K$ is continuous.
By the Tietze extension theorem, there exists $\tilde\rho \in C(\bT^d)$ with 
\[
\tilde{\rho} |_K = \rho |_K, \quad 
\norm{\tilde\rho}_{L^\infty} \le \norm{\rho}_{L^\infty}.
\]
For any $f \in L^2$, one has that 
\begin{multline}\label{eq:young_three_terms}
\left| \int_{\bT^d} (\rho \circ \Psi_{j_k})(x) f(x) \,\rd x - 
\int_{\bT^d} \langle \nu^x, \rho \rangle f(x) \,\rd x \right| 
\le \left| \int_{\bT^d} (\tilde\rho \circ \Psi_{j_k})(x) f(x) \,\rd x - 
\int_{\bT^d} \langle \nu^x, \tilde\rho \rangle f(x) \,\rd x \right| \\+ 
\norm{\rho \circ \Phi_{j_k} - \tilde\rho \circ \Phi_{j_k}}_{L^2} 
\norm{f}_{L^2} + 
\norm{\langle \nu^x, \rho \rangle - \langle \nu^x, \tilde \rho 
\rangle}_{L^2} 
\norm{f}_{L^2}.
\end{multline}
Since $\tilde\rho \in C(\bT^d)$, one can obtain by \Cref{thm:ball} that 
\[
\tilde\rho \circ \Psi_{j_k} \rightharpoonup 
\langle \nu^x, \tilde\rho \rangle \quad \text{weakly in } L^2(\bT^d),
\]
which yields an estimate of the first term on the right-hand side of 
\cref{eq:young_three_terms} as follows,
\[
\left| \int_{\bT^d} (\tilde\rho \circ \Psi_{j_k})(x) f(x) \,\rd x - 
\int_{\bT^d} \langle \nu^x, \tilde\rho \rangle f(x) \,\rd x \right| < 
\epsilon,
\]
as long as $k$ is sufficiently large. 
For the second term, since each $\Psi_{j_k}$ preserves 
Lebesgue measure, \cref{eq:young_lusin} yields that 
\[
\norm{\rho \circ \Phi_{j_k} - \tilde\rho \circ \Phi_{j_k}}_{L^2}^2 = 
\norm{\rho - \tilde\rho}_{L^2}^2 = \int_{\bT^d \setminus K} 
|\rho - \tilde\rho|^2 \,\rd x \le 4 \epsilon \norm{\rho}_{L^\infty}^2.
\]
For the third term, one obtains by \cref{eq:young_set,eq:young_lusin} that 
\begin{multline*}
\norm{\langle \nu^x, \rho \rangle - \langle \nu^x, \tilde \rho 
\rangle}_{L^2} =  \left (\int_{\bT^d} \left| \int_{\bT^d} 
(\rho - \tilde\rho) \,\rd \nu^x \right|^2 \,\rd x \right)^{1/2} \le 
2 \norm{\rho}_{L^\infty} \left| \int_{\bT^d} \nu^x(\bT^d \setminus K)^2 
\,\rd x \right|^{1/2} \\
\le 2 \norm{\rho}_{L^\infty} \left| \int_{\bT^d} \nu^x(\bT^d \setminus K) 
\,\rd x \right|^{1/2} \le 2 \norm{\rho}_{L^\infty} 
\mu(\bT^d \setminus K)^{1/2} \le 2 \epsilon^{1/2} \norm{\rho}_{L^\infty}.
\end{multline*}
Combining these estimates, one obtains the following estimate of the left-hand 
term of \cref{eq:young_three_terms} as follows,
\[
\left| \int_{\bT^d} (\rho \circ \Psi_{j_k})(x) f(x) \,\rd x - 
\int_{\bT^d} \langle \nu^x, \rho \rangle f(x) \,\rd x \right| 
\le \epsilon + 4 \epsilon^{1/2} \norm{\rho}_{L^\infty},
\]
which can be made arbitrarily small as long as $k$ is sufficiently large.
This establishes the weak convergence in $L^2$ and completes the proof of this 
theorem.
\end{proof}

\subsection{Young-measure characterization of mixing}\label{sec:young_cha}
Recall from Definition~\ref{def:mixing} that an initial datum 
$\rho\in L^\infty(\mathbb{T}^d)$ is \emph{not mixed} by $u$ if and only if the 
trajectory $\{\rho \circ \Phi_t^{-1}\}_{t \ge 0}$ is precompact in $L^2$.
In order to reformulate this in terms of Young measures, we first define:

\begin{definition}
Let $u \in L^\infty (W^{1,p})$ be a divergence-free velocity field, and let 
$\{\Phi_t\}_{t \ge 0}$ be the associated measure-preserving flow maps.
We say a family of probability measures $\{\nu^x\}_{x \in \bT^d}$ is a 
\emph{Young measure generated by $u$}, if 
there exists a sequence $\{t_j\}_{j \in \bN}$ with $t_j\to\infty$ such that 
the inverse maps $\Phi_{t_j}^{-1}$ generate $\{\nu^x\}_{x \in \bT^d}$ 
in the sense of \Cref{thm:youngL}, i.e., 
\[
\rho \circ \Phi_{t_j}^{-1} \rightharpoonup 
\langle \nu^x, \rho \rangle \quad \text{weakly in } L^2(\bT^d),
\]
for every $\rho \in L^\infty(\bT^d)$.
At this time, we say 
\[
\Phi_{t_j}^{-1} \to \{\nu^x\}_{x \in \bT^d} \quad 
\text{in the Young-measure sense}.
\]
The set of all such Young measures $\{\nu^x\}_{x \in \bT^d}$ generated by 
the velocity field $u$ is denoted by $\cY_u$.
\end{definition}

\begin{lemma}\label{lemma:young}
Let $u \in L^\infty (W^{1,p})$ be a divergence-free velocity field, 
$\{\Phi_t\}_{t \ge 0}$ be its flow maps, and $\rho \in L^\infty(\bT^d)$.
Then the following conditions are equivalent:
\begin{enumerate}
\item $\rho \not \in \cF_u$, i.e., the trajectory 
$\{\rho \circ \Phi_t^{-1}\}_{t \ge 0}$ is precompact in $L^2$;
\item For every Young measure $\{ \nu^x \}_{x \in \bT^d} \in \cY_u$ generated 
by $u$, one has that 
\begin{equation}\label{eq:lemma_young_cond2}
\rho(y) = \int_{\bT^d} \rho(z) \,\rd \nu^x(z) \quad 
\text{for $\mu$-a.e. $x$ and $\nu^x$-a.e. $y$}. 
\end{equation}
\end{enumerate}
\end{lemma}

\begin{proof}
Fix any sequence $\{t_j\}_{j \in \bN}$ with $t_j \to \infty$ and 
extract a subsequence $\{ t_{j_k} \}_{k \in \bN}$ such that 
$\Phi_{t_{j_k}}^{-1} \to \{\nu^x\}_{x \in \bT^d}$ in the Young-measure 
sense.
By \Cref{thm:youngL}, this yields that 
\begin{equation}\label{eq:lemma_young_weak}
\rho \circ \Phi_{t_{j_k}}^{-1} \rightharpoonup 
\langle \nu^x, \rho \rangle, \quad 
|\rho|^2 \circ \Phi_{t_{j_k}}^{-1} \rightharpoonup 
\langle \nu^x, |\rho|^2 \rangle
\quad \text{weakly in } L^2.
\end{equation}
Taking the test function $f \equiv 1$ gives that 
\[
\norm{\rho}_{L^2}^2 = \norm{\rho \circ \Phi_{t_{j_k}}^{-1}}_{L^2}^2 \to 
\int_{\bT^d} \langle \nu^x, |\rho|^2 \rangle \,\rd x.
\]
By Jensen's inequality, 
\[
\left| \langle \nu^x, \rho \rangle \right|^2 \le 
\langle \nu^x, |\rho|^2 \rangle,
\]
with equality if and only if $\rho$ is $\nu^x$-a.e. constant.
Integrating over $x$ yields that 
\begin{equation}\label{eq:lemma_young_int}
\norm{\langle \nu^x, \rho \rangle}_{L^2}^2 \le 
\int_{\bT^d} \langle \nu^x, |\rho|^2 \rangle \,\rd x 
= \norm{\rho}_{L^2}^2,
\end{equation}
and the equality holds if and only if for $\mu$-a.e. $x \in \bT^d$, 
$\rho(y)$ is constant for $\nu^x$-a.e. $y \in \bT^d$.

Now suppose first that $\{\rho \circ \Phi_t^{-1}\}_{t \ge 0}$ is precompact in 
$L^2$. At this time, the convergence in \cref{eq:lemma_young_weak} should be 
strong convergence in $L^2$, which implies convergence of $L^2$-norms.
Thus, the equality of \cref{eq:lemma_young_int} holds, yielding 
\cref{eq:lemma_young_cond2}.
Conversely, assume that \cref{eq:lemma_young_cond2} holds for each 
$\{\nu^x\}_{x \in \bT^d} \in \cY_u$, which yields that the equality holds in  
\cref{eq:lemma_young_int}. 
Therefore, $\rho \circ \Phi_{t_{j_k}}^{-1}$ converges to 
$\langle \nu^x, \rho \rangle$ strongly in $L^2$ as $k \to \infty$.
Note the sequence $\{ t_j \}_{j \in \bN}$ is arbitrary. 
Hence $\{ \rho \circ \Phi_t^{-1} \}_{t \ge 0}$ is precompact in $L^2$.
\end{proof}

\section{Structure of initial data}\label{sec:structure}
\Cref{thm:nong} demonstrates that mixing for \emph{all} nontrivial initial 
data cannot persist under arbitrarily small perturbations in the 
$L^\infty (W^{1,p})$ topology.
In light of this non-robustness, it is natural to ask whether mixing for 
\emph{generic} initial data can nevertheless be a generic property of 
incompressible flows. 
The objective of this section is to gain more understanding of the velocity 
fields that mix generic initial data.

By employing the Young–measure characterization of mixing from 
Lemma~\ref{lemma:young}, we identify the $\sigma$-algebra structure of unmixed 
level sets in \Cref{sec:structure_sigma}, and characterize the precise 
structure of those initial data that are mixed by a given divergence-free 
velocity field in \Cref{sec:structure_mix}.
Based on this, we prove in \Cref{sec:structure_generic}, if one initial datum 
is mixed by $u$, then the set of all the initial data that are mixed by $u$ is 
open and dense in $L^\infty$.

\subsection{Unmixed $\sigma$-algebra}\label{sec:structure_sigma}
Let us first study the structure of those unmixed initial data taking 
form of characteristic functions.
Define the set family 
\[\cN_u := \left\{ D \in \cM \,\big|\, \cha_D \not \in \cF_u \right\}.\]

\begin{remark}\label{rmk:structure_young}
By Lemma~\ref{lemma:young}, the element in $\cN_u$ can be characterized in 
terms of Young measures: $D \in \cN_u$ if and only if for every Young 
measure $\{\nu^x\}_{x \in \bT^d} \in \cY_u$, one has that 
\[
\nu^x(D) = 1 \quad \text{or} \quad \nu^x(D) = 0,
\]
for $\mu$-a.e. $x \in \bT^d$.
\end{remark}

\begin{lemma}\label{lemma:sigma}
The set family $\cN_u$ is a $\sigma$-algebra.
We call $\cN_u$ the \emph{unmixed $\sigma$-algebra} of $u$.
\end{lemma}

\begin{proof}
First, $\cha_{\varnothing} \circ \Phi_t^{-1} \equiv 0 \not \in \cF_u$, 
so $\varnothing \in \cN_u$, and complementation is preserved since 
$\cha_{D^\complement} \circ \Phi_t^{-1} = 1 - \cha_D \circ \Phi_t^{-1}$.

Next, let $\{D_k\}_{k \in \bN} \subset \cN_u$ be a countable family of 
disjoint sets, and define $D := \cup_{k\in\bN} D_k$.
Fix any Young measures $\{\nu^x\}_{x \in \bT^d} \in \cY_u$.
By Remark~\ref{rmk:structure_young}, for $\mu$-a.e. $x \in \bT^d$, 
if $\nu^x(D_k) = 0$ for every $k \in \bN$, then $\nu^x(D) = 0$;
otherwise, if there exists $k \in \bN$ such that $\nu^x(D_k) = 1$, then 
$\nu^x(D) = 1$. Thus, $D \in \cN_u$.
Therefore, $\cN_u$ is a $\sigma$-algebra.
\end{proof}

\subsection{Mixed initial data}\label{sec:structure_mix}
We now describe precisely which bounded measurable functions fail to mix by 
a given divergence-free velocity field $u$. 
In particular, we show that $\cN_u$ governs the $\sigma$-algebra of the 
level sets of any unmixed~$\rho$.

\begin{theorem}\label{thm:cha}
A function $\rho\in L^\infty(\mathbb{T}^d)$ satisfies 
$\rho\notin\cF_u$ if and only if $\rho$ is measurable with respect to 
the $\sigma$-algebra $\mathcal{N}_u$. Equivalently, the set of mixed initial 
data can be characterized by 
\[
\cF_u = \left\{ \rho \in L^\infty(\bT^d) \mid \rho \text{ is not } \cN_u 
\text{-measurable} \right\}.
\]
\end{theorem}

\begin{proof}
Suppose first that $\rho \in L^\infty$ is $\mathcal{N}_u$-measurable. 
By standard measure-theoretic arguments, there exists a sequence of simple 
functions
\[
\rho_n = \sum_{j = 1}^{N_n} a_{n,j} \cha_{D_{n,j}}, 
\]
with $D_{n,j} \in \cN_u$ and 
$\norm{\rho_n}_{L^\infty} \le \norm{\rho}_{L^\infty}$, such that 
$\rho_n \to \rho$ for each $x \in \bT^d$ as $n \to \infty$.
Note that $\rho_n \not\in \cF_u$ for each $n \in \bN$. 
By Lemma~\ref{lemma:young}, for every $\{\nu^x\}_{x \in \bT^d} \in \cY_u$, 
one has that 
\[
\rho_n(y) = \int_{\bT^d} \rho_n(z) \,\rd \nu^x(z) \quad 
\text{for $\mu$-a.e. $x$ and $\nu^x$-a.e. $y$}. 
\]
By the dominated convergence theorem and taking the limit $n \to \infty$, 
one has that
\[
\rho(y) = \int_{\bT^d} \rho(z) \,\rd \nu^x(z) \quad 
\text{for $\mu$-a.e. $x$ and $\nu^x$-a.e. $y$},
\]
which yields that $\rho \notin \cF_u$ by Lemma~\ref{lemma:young}.

Conversely, fix $\rho \not \in \cF_u$ and $\alpha \in \bR$.
By Lemma~\ref{lemma:young}, for every $\{ \nu^x \}_{x \in \bT^d} \in \cY_u$,
one has that 
\[
\rho(y) = \int_{\bT^d} \rho(z) \,\rd \nu^x(z) \quad 
\text{for $\mu$-a.e. $x$ and $\nu^x$-a.e. $y$}. 
\]
Denote the level set by
$D_\alpha = \left\{ x \in \bT^d \mid \rho(x) \ge \alpha \right\}$.
If $\int_{\bT^d} \rho \,\rd \nu^x \ge \alpha$, then $\rho(y) \ge \alpha$ 
for $\nu^x$-a.e. $y$, and thus, $\nu^x(D_\alpha) = 1$;
if $\int_{\bT^d} \rho \,\rd \nu^x < \alpha$, then $\rho(y) < \alpha$ 
for $\nu^x$-a.e. $y$, and thus, $\nu^x(D_\alpha) = 0$.
By Remark~\ref{rmk:structure_young}, $D_\alpha \in \cN_u$.
Therefore, $\rho$ is $\cN_u$-measurable.
This completes the proof.
\end{proof}

\begin{remark}
Essentially, \Cref{thm:cha} asserts that the trajectory 
$\{\rho \circ \Phi_t^{-1}\}_{t \ge 0}$ 
is precompact in $L^2$ if and only if, for every $\alpha \in \bR$, 
the characteristic functions of level sets,
\[
\left\{ \cha_{\{\rho \ge \alpha\}} \circ \Phi_t^{-1} 
= \cha_{\{\rho \circ \Phi_t^{-1} \ge \alpha\}} \right\}_{t \ge 0},
\]
is precompact in $L^2$.
This equivalence relies crucially on the fact that the measure-preserving flow 
maps $\Phi_t$ leaves the distribution of $\rho \circ \Phi_t^{-1}$ invariant for 
each $t \ge 0$.
By contrast, for a general function family $\cA \subset L^\infty(\bT^d)$, 
precompactness of $\cA$ in $L^2$ does \emph{not} imply that each collection of 
level-set indicators $\{\cha_{\{\rho \ge \alpha\} \mid \rho \in \cA}\}$ is 
precompact in $L^2$.
For example, let 
\[
\rho_k(x) := \frac{1}{k} \prod_{j = 1}^d \sin (2 \pi k x_j), \quad x \in \bT^d, 
\ k \in \bN.
\]
At this time, $\rho_k \to 0$ strongly in $L^2$,
yet each level set $\{\rho_k \ge 0\}$ has Lebesgue measure $\frac{1}{2}$, and 
$\cha_{\{\rho_k \ge 0\}} \rightharpoonup 0$ only weakly in $L^2$.
Hence $\{\cha_{\{ \rho_k \ge 0 \}}\}_{k \in \bN}$ fails to be precompact in 
$L^2$, even though $\{\rho_k\}_{k \in \bN}$ itself is precompact.
\end{remark}

\subsection{Generic mixing}\label{sec:structure_generic}
\Cref{thm:cha} implies that $u$ is universally non-precompact 
(i.e. every nontrivial $\rho \in L^\infty$ mixes), if and only if 
$\cN_u$ is the $\sigma$-algebra generated by the null sets $\cJ$.
As for mixing for generic initial data, we have the following corollary
of \Cref{thm:cha}.

\begin{corollary}\label{cor:generic}
The set of mixed initial data, $\cF_u$, is open and dense in $L^\infty$, 
if and only if $\cN_u \not = \cM$.
\end{corollary}

\begin{proof}
Take a sequence $\{\rho_j\}_{j \in \bN}$ of $\cN_u$-measurable functions 
such that $\rho_j \to \rho$ in $L^\infty$, which yields that $\rho$ is also 
$\cN_u$-measurable.
Therefore, the set of all the $\cN_u$-measurable functions is closed in 
$L^\infty$.
Suppose that $\cN_u \not = \cM$. 
Take a set $D \in \cM \setminus \cN_u$.
Fix $\rho \in L^\infty$ and $\epsilon > 0$.
If $\rho$ is $\cN_u$-measurable, then $\rho + \epsilon \cha_{D}$ is 
$\cN_u$-unmeasurable.
Therefore, $\cF_u$ is dense in $L^\infty$ at this time.
Conversely, it is straightforward that $\cF_u = \varnothing$ 
if $\cN_u = \cM$.
\end{proof}

\begin{remark}
By Theorem~\ref{thm:cha} together with Corollary~\ref{cor:generic}, 
one obtains the following: 
if there exists a single initial datum $\rho_0 \in L^\infty$ which is mixed by 
$u$, then the full set of mixed initial data, $\cF_u$, is open and dense in 
$L^\infty$.
We remark that the denseness of $\cF_u$ in $L^\infty$ is straightforward, 
whereas its openness relies critically on the $\sigma$-algebra characterization 
provided by Theorem~\ref{thm:cha}.
\end{remark}

\begin{remark}
The phenomenon that incompressible flows mix a \emph{generic} set of initial 
data, rather than \emph{all} nontrivial data, also appears in the 
two-dimensional incompressible Euler equations.
In fact, \v{S}ver\'ak has conjectured 
\cite{drivas2023singularity,sverak2011course}
that for a generic vorticity $\omega_0 \in L^\infty(\bT^2)$, the corresponding  
inviscid incompressible Euler flow has non-precompact vorticity orbit 
$\{\omega(t, \cdot)\}_{t \ge 0}$ in $L^2(\bT^d)$.
\end{remark}

\section{Topology of measure-preserving flow maps}\label{sec:topo}
In this section, we introduce a natural topology on the group of all 
measure-preserving bijections on~$\bT^d$, and relate the non-precompactness of 
the flow maps $\{\Phi_t\}_{t\ge0}$ generated by $u$ to the $\sigma$-algebra 
$\mathcal{N}_u$ of unmixed sets defined in \Cref{sec:structure}. 
Throughout, let $\Aut(\bT^d,\mu)$ denote the collection of all 
measure-preserving bijections on~$\bT^d$ that preserves the Lebesgue measure 
$\mu$.

\subsection{Convergence in symmetric difference}
We endow $\Aut(\bT^d, \mu)$ with the topology induced by the metric 
\[
d(\Phi, \Psi) := \sup_{D \in \cM} \mu (\Phi(D) \triangle \Psi(D)),
\]
where $\triangle$ denotes symmetric difference.
Equivalently, a sequence $\Phi_j$ converges to~$\Phi$ in this topology 
if and only if 
\begin{equation}\label{eq:sym_diff}
\lim_{j \to \infty} \mu (\Phi_j(D) \triangle \Phi(D)) = 0 \quad 
\text{for every } D \in \cM.
\end{equation}

\begin{theorem}\label{thm:topo}
The family of flow maps $\{\Phi_t\}_{t \ge 0} \subset \Aut(\bT^d, \mu)$ is 
precompact in the metric $d(\cdot, \cdot)$ if and only if $\cN_u = \cM$.
\end{theorem}

\begin{proof}
Suppose first that $\{\Phi_t\}_{t \ge 0}$ is precompact in $\Aut(\bT^d, \mu)$. 
Fix any sequence $\{t_j\}_{j \in \bN}$ with $t_j \to \infty$.
By precompactness, there exists a subsequence $\{t_{j_k}\}_{k \in \bN}$,
and a limit $\Phi \in \Aut(\bT^d, \mu)$ such that 
\[
\lim_{k \to \infty} d(\Phi_{t_{j_k}}, \Phi) = 0.
\]
In particular, for each measurable set $D \in \cM$, 
\[
\norm{\cha_D \circ \Phi_{t_{j_k}}^{-1} - \cha_D \circ \Phi^{-1}}_{L^2} = 
\mu(\Phi_{t_{j_k}}(D) \triangle \Phi(D))^{1/2} \to 0,
\]
which yields that $\cha_D \not \in \cF_u$. 
Since $D \in \cM$ is arbitrary, it follows that $\cN_u = \cM$.

Conversely, assume that $\cN_u = \cM$.
Fix any sequence $\{ t_j \}_{j \in \bN}$ with $t_j \to \infty$ such that 
$\Phi_{t_j}^{-1} \to \{\nu^x\}_{x \in \bT^d}$ in the Young-measure sense.
For each $D \in \cM$, the sequence $\cha_D \circ \Phi_{t_j}^{-1}$
converges strongly in $L^2$, since $\cha_D \not \in \cF_u$.
In particular, it must converge to a characteristic function 
$\cha_{\tilde{D}}$, which is determined up to null sets.
At this time, the map $D \mapsto \tilde{D}$ is a 
$\sigma$-isomorphism from $\cM / \cJ$ to $\cM / \cJ$, where $\cJ$ is the 
ideal of null sets, and $\tilde{D}$ depends on $D$.
By \cite[Theorem~15.10]{Kechris1995Classical}, there exists a 
measure-preserving bijection $\Phi \in \Aut(\bT^d, \mu)$, uniquely defined 
up to null sets, realizing this $\sigma$-isomorphism, that is, 
$\Phi(D) = \tilde{D}$ a.e. for each $D \in \cM$.
Consequently, for each $D \in \cM$, one has that 
\[
\mu( \Phi_{t_j}(D) \triangle \Phi(D) ) = 
\norm{\cha_{D} \circ \Phi_{t_j}^{-1} - \cha_D \circ \Phi^{-1}}_{L^2}^2 \to 0,
\]
which implies that $d(\Phi_{t_j}, \Phi) \to 0$.
By the fundamental theorem of Young measures, for each sequence 
$\{t_j\}_{j \in \bN}$ such that $t_j \to \infty$, there exists a subsequence 
$\{t_{j_k}\}_{k \in \bN}$ such that $\Phi_{t_{j_k}}$ converges in the metric 
$d(\cdot, \cdot)$.
Therefore, the flow maps $\{\Phi_t\}_{t \ge 0}$ is precompact in 
$\Aut(\bT^d, \mu)$.
\end{proof}

\subsection{Convergence in measure}
In their foundational work \cite{DiPerna1989Ordinary}, DiPerna and Lions 
endowed the group $\Aut(\bT^d, \mu)$ with the topology induced by the $L^1$ 
metric 
\[
d_{\mathrm{DL}}(\Phi, \Psi) = 
\norm{\Phi - \Psi}_{L^1(\bT^d)}.
\]
which is equivalent to the topology induced by convergence in measure:
\[
d_{\mathrm{DL}}(\Phi_j, \Phi) \to 0 \ \Longleftrightarrow\ 
\mu \left\{ x \in \bT^d \mid 
|\Phi_j(x) - \Phi(x)| > \delta \right\} \to 0 \quad \text{for each } \delta > 0.
\]
We now verify that this convergence-in-measure topology coincides with the 
symmetric-difference topology induced by~$d(\cdot, \cdot)$ in 
\cref{eq:sym_diff}.

\begin{proposition}\label{prop:topo}
For any sequence $\{\Phi_k\}_{k \in \bN} \subset \Aut(\bT^d, \mu)$ and any 
$\Phi \in \Aut(\bT^d, \mu)$, one has that 
\[
d_{\mathrm{DL}}(\Phi_k, \Phi) \to 0 \ \Longleftrightarrow \ 
d(\Phi_k, \Phi) \to 0.
\]
\end{proposition}

\begin{proof}
First, assume that $d(\Phi_k, \Phi) \to 0$. 
Suppose, for contradiction, that $\Phi_k$ does not converge to $\Phi$ in 
measure.
Then there exist $\delta > 0$, $\eta > 0$, and a subsequence 
(still denoted by $\{\Phi_k\}_{k \in \bN}$) 
such that for all~$k$, one has that 
\[
\mu \left\{ x \in \bT^d \,\big|\, 
|\Phi_k \circ \Phi^{-1}(x) - x| > \delta \right\} = 
\mu \left\{ x \in \bT^d \,\big|\, |\Phi_k(x) - \Phi(x)| > \delta 
\right\} > \eta.
\]
Define 
\[
A_k := \left\{ x \in \bT^d \,\big|\, |\Phi_k \circ \Phi^{-1}(x) - x| > \delta 
\right\},
\]
and 
\[
A : = \bigcap_{m \in \bN} \bigcup_{k \ge m} A_k.
\]
Since $\mu (\bT^d) < \infty$, one has that 
\[
\mu(A) = \lim_{m \to \infty} \mu (\cup_{k \ge m} A_k) \ge \eta > 0.
\]
Passing to a further subsequence, still denoted by $\{\Phi_k\}_{k \in \bN}$, 
we may assume that 
\[
|\Phi_k \circ \Phi^{-1}(x) - x| > \delta \quad \text{for each } x \in A, 
k \in \bN.
\]
By the Lebesgue density theorem, there exists a ball 
$B_\delta(x_0) \subset \bT^d$ with $\delta_0 \in (0, \delta / 4)$ such that 
\[
\mu(A_0) > 
\frac{1}{2} \mu(B_{\delta_0}(x_0)) > 0, \quad A_0 := A \cap B_{\delta_0}(x_0).
\]
For each $k \in \bN$ and each $x, y \in A_0$, 
\[
|\Phi_k \circ \Phi^{-1}(x) - y| \ge 
|\Phi_k \circ \Phi^{-1}(x) - x| - |x - y| \ge 
\delta - 2 \delta_0 > 2 \delta_0 > 0,
\]
which yields that $A_0 \cap \Phi_k \circ \Phi^{-1}(A_0) = \varnothing$.
Therefore, for each $k \in \bN$, one has that 
\[
\mu \left( \Phi_k(A_0) \triangle \Phi(A_0) \right) = 
\mu \left( \Phi_k \circ \Phi^{-1}(A_0) \triangle A_0 \right) = 
2 \mu(A_0) > 0,
\]
contradicting $d(\Phi_k, \Phi) \to 0$. 
This proves convergence in measure.

Conversely, assume that $\Phi_k \to \Phi$ in measure. 
By the definition of the metric $d$, it suffices to prove 
\[
\mu(\Phi_k(B) \triangle \Phi(B)) \to 0 \quad \text{for every ball } 
B \subset \bT^d,
\]
since it is possible to approximate any Lebesgue measurable set in $\bT^d$ 
by a finite disjoint union of balls such that the symmetric difference has 
arbitrarily small measure.
Fix a ball $B = B_r(x_0)$ and let $\epsilon \in (0, \frac{r}2)$.
By convergence in measure, there exists $K \in \bN$ such that for all 
$k \ge K$, 
\[
\mu (D_k) < \varepsilon, \quad D_k := \left\{ x \in \bT^d \,\big|\, 
|\Phi_k \circ \Phi^{-1}(x) - x| > \varepsilon \right\}
\]
If $x \in B_{r - \varepsilon}(x_0) \setminus D_k$, then 
\[
|\Phi_k \circ \Phi^{-1}(x) - x_0| \le 
|\Phi_k \circ \Phi^{-1}(x) - x| + |x - x_0| < \epsilon + (r - \epsilon) = r,
\]
so at this time, $\Phi_k \circ \Phi^{-1}(x) \in B_r(x_0)$.
Consequently, 
\[
\Phi_k \circ \Phi^{-1} (B_{r-\epsilon}(x_0) \setminus D_k) \subset B_r(x_0),
\]
which yields that 
\begin{multline*}
B_r(x_0) \setminus \Phi_k \circ \Phi^{-1} (B_r(x_0)) \\ \subset 
B_r(x_0) \setminus (B_{r-\epsilon}(x_0) \setminus D_k) = 
(B_r(x_0) \setminus B_{r-\epsilon}(x_0)) \cup (B_r(x_0) \cap D_k).
\end{multline*}
Therefore, 
\[
\mu(B_r(x_0) \setminus \Phi_k \circ \Phi^{-1} (B_r(x_0))) \le 
\mu(B_r(x_0) \setminus B_{r-\epsilon}(x_0)) + \mu(D_k) \le 
\alpha_d (r^d - (r- \epsilon)^d) + \epsilon,
\]
where $\alpha_d$ is the volume of the $d$-dimensional unit ball.
It follows that 
\[
\mu(\Phi_k(B_r(x_0)) \triangle \Phi(B_r(x_0)) ) = 2 
\mu(B_r(x_0) \setminus \Phi_k \circ \Phi^{-1} (B_r(x_0))) \to 0
\quad \text{ as } k \to \infty,
\]
as claimed.
This completes the proof.
\end{proof}

\begin{remark}
DiPerna and Lions originally worked on the whole space $\bR^d$ 
\cite{DiPerna1989Ordinary}, 
in which case they equipped $\Aut(\bR^d, \mu)$ with the metric
\[
\tilde{d}_{\mathrm{DL}}(\Phi, \Psi) = \sum_{n \ge 1} \frac{1}{2^n} 
\norm{\min \{ |\Phi - \Psi|, 1 \}}_{L^1(B_n(0))},
\]
so that convergence in $\tilde{d}_{\mathrm{DL}}$ is equivalent to convergence 
in measure on arbitrary balls.
Although this convergence is not the same as convergence in measure on 
$\bR^d$, a  modification of the proof of Proposition~\ref{prop:topo} shows 
that convergence in $\tilde{d}_{\mathrm{DL}}$ is nevertheless equivalent to 
convergence in symmetric difference, in the sense of \eqref{eq:sym_diff}.
\end{remark}

\begin{remark}
Since $\bT^d$ is compact with $\mu(\bT^d) < \infty$, 
Proposition~\ref{prop:topo} implies that, for any sequence of 
measure-preserving maps on $\bT^d$,
convergence in symmetric difference as in \cref{eq:sym_diff}
is equivalent to convergence in measure, and 
hence to convergence in $L^p$ for every $1 \le p < \infty$.
The symmetric-difference topology captures the asymptotic behavior of 
$\rho(t, \cdot)$, whereas the $L^p$ topology reflects properties of the flow 
maps themselves. 
In this sense, Proposition~\ref{prop:topo} provides a crucial link between 
these perspectives and suggests avenues for further investigation.
\end{remark}

\section{Conclusions}\label{sec:concl}

We have shown that the classical mixing criteria, particularly, mixing for all 
nontrivial initial data, fail to persist under arbitrarily small perturbations 
of a divergence-free velocity field in $L^\infty (W^{1,p})$.
To address this lack of robustness, we introduce a Young–measure framework and 
prove that a bounded density fails to mix if and only if each of its level sets 
remains unmixed.
This yields an associated unmixed $\sigma$-algebra,  
whose measurability precisely characterizes the failure of mixing. 
In particular, we establish that 
``one datum mixes'' is equivalent to ``generic data mix'', which in turn is 
equivalent to the non-precompactness of the measure-preserving flow maps 
in the $L^1$ topology.

A natural question arises: whether, in the Baire-generic sense, divergence-free 
velocity fields necessarily fail to generate precompact trajectories.
We therefore propose:

\begin{conjecture}\label{conj}
Fix $p \in (1, \infty)$. For a residual subset of 
\[
\{u \in L^\infty([0,+\infty); W^{1,p}(\bT^d; \bR^d)) \mid \nabla \cdot u = 0\},
\]
the associated flow maps $\{\Phi_t\}_{t \ge 0}$ are not precompact in $L^1$.
\end{conjecture}

In the present work, we have analyzed the genericity of mixing for a given 
incompressible flow in terms of the initial data of the transport equation. 
Conjecture~\ref{conj} addresses the complementary question of the genericity of 
divergence-free velocity fields that do mix passive scalars. 
An affirmative resolution would complete the description of typical mixing 
behavior in incompressible fluids.

\section*{Acknowledgements}
The authors would like to thank 
Dr. Yao Yao in Department of Mathematics, National University of Singapore, and 
Dr. Yizhou Zhou in IGPM, RWTH Aachen University, for insightful discussions.


\begin{thebibliography}{10}

\bibitem{alberti2014uniqueness}
{\sc Alberti, G., Bianchini, S., and Crippa, G.}
\newblock A uniqueness result for the continuity equation in two dimensions.
\newblock {\em J. Eur. Math. Soc. (JEMS) 16}, 2 (2014), 201--234.

\bibitem{albertiExponentialSelfsimilarMixing2018a}
{\sc Alberti, G., Crippa, G., and Mazzucato, A.~L.}
\newblock Exponential self-similar mixing by incompressible flows.
\newblock {\em J. Amer. Math. Soc. 32}, 2 (2019), 445--490.

\bibitem{Alberti2019Loss}
{\sc Alberti, G., Crippa, G., and Mazzucato, A.~L.}
\newblock Loss of regularity for the continuity equation with non-{L}ipschitz
  velocity field.
\newblock {\em Ann. PDE 5}, 1 (2019), Paper No. 9, 19.

\bibitem{Ambrosio2004Transport}
{\sc Ambrosio, L.}
\newblock Transport equation and {C}auchy problem for {$BV$} vector fields.
\newblock {\em Invent. Math. 158}, 2 (2004), 227--260.

\bibitem{ambrosioContinuityEquationsODE2014}
{\sc Ambrosio, L., and Crippa, G.}
\newblock Continuity equations and {ODE} flows with non-smooth velocity.
\newblock {\em Proc. Roy. Soc. Edinburgh Sect. A 144}, 6 (2014), 1191--1244.

\bibitem{Hassan1984stirring}
{\sc Aref, H.}
\newblock Stirring by chaotic advection.
\newblock {\em J. Fluid Mech. 143\/} (1984), 1--21.

\bibitem{BALDYGA1997457}
{\sc Ba{\l}dyga, J., Bourne, J., and Hearn, S.}
\newblock Interaction between chemical reactions and mixing on various scales.
\newblock {\em Chemical Engineering Science 52}, 4 (1997), 457--466.

\bibitem{Ball1989Young}
{\sc Ball, J.~M.}
\newblock A version of the fundamental theorem for {Y}oung measures.
\newblock In {\em P{DE}s and continuum models of phase transitions ({N}ice,
  1988)}, vol.~344 of {\em Lecture Notes in Phys.} Springer, Berlin, 1989,
  pp.~207--215.

\bibitem{Bedrossian2016enhanced}
{\sc Bedrossian, J., Masmoudi, N., and Vicol, V.}
\newblock Enhanced dissipation and inviscid damping in the inviscid limit of
  the {N}avier-{S}tokes equations near the two dimensional {C}ouette flow.
\newblock {\em Arch. Ration. Mech. Anal. 219}, 3 (2016), 1087--1159.

\bibitem{bessaGenericIncompressibleFlow2008}
{\sc Bessa, M.}
\newblock A generic incompressible flow is topological mixing.
\newblock {\em C. R. Math. Acad. Sci. Paris 346}, 21-22 (2008), 1169--1174.

\bibitem{boffettaNonasymptoticPropertiesTransport2000}
{\sc Boffetta, G., Celani, A., Cencini, M., Lacorata, G., and Vulpiani, A.}
\newblock Nonasymptotic properties of transport and mixing.
\newblock {\em Chaos 10}, 1 (2000), 50--60.
\newblock Chaotic kinetics and transport (New York, 1998).

\bibitem{Bogovskii1980solutions}
{\sc Bogovski\u{i}, M.~E.}
\newblock Solutions of some problems of vector analysis, associated with the
  operators {${\rm div}$}\ and {${\rm grad}$}.
\newblock In {\em Theory of cubature formulas and the application of functional
  analysis to problems of mathematical physics}, vol.~No. 1, 1980 of {\em Proc.
  Sobolev Sem.} Akad. Nauk SSSR Sibirsk. Otdel., Inst. Mat., Novosibirsk, 1980,
  pp.~5--40, 149.

\bibitem{bressanLemmaConjectureCost2003}
{\sc Bressan, A.}
\newblock A lemma and a conjecture on the cost of rearrangements.
\newblock {\em Rend. Sem. Mat. Univ. Padova 110\/} (2003), 97--102.

\bibitem{constantinDiffusionMixingFluid2008}
{\sc Constantin, P., Kiselev, A., Ryzhik, L., and Zlato\v{s}, A.}
\newblock Diffusion and mixing in fluid flow.
\newblock {\em Ann. of Math. (2) 168}, 2 (2008), 643--674.

\bibitem{cotizelatiMixingIncompressibleFlows2024}
{\sc Coti~Zelati, M., Crippa, G., Iyer, G., and Mazzucato, A.~L.}
\newblock Mixing in incompressible flows: transport, dissipation, and their
  interplay.
\newblock {\em Notices Amer. Math. Soc. 71}, 5 (2024), 593--604.

\bibitem{coti2020relation}
{\sc Coti~Zelati, M., Delgadino, M.~G., and Elgindi, T.~M.}
\newblock On the relation between enhanced dissipation timescales and mixing
  rates.
\newblock {\em Comm. Pure Appl. Math. 73}, 6 (2020), 1205--1244.

\bibitem{Crippa2008Estimates}
{\sc Crippa, G., and De~Lellis, C.}
\newblock Estimates and regularity results for the {D}i{P}erna-{L}ions flow.
\newblock {\em J. Reine Angew. Math. 616\/} (2008), 15--46.

\bibitem{cullen2009food}
{\sc Cullen, P.~J.}
\newblock {\em Food mixing: Principles and applications}.
\newblock John Wiley \& Sons, 2009.

\bibitem{DiPerna1989Ordinary}
{\sc DiPerna, R.~J., and Lions, P.-L.}
\newblock Ordinary differential equations, transport theory and {S}obolev
  spaces.
\newblock {\em Invent. Math. 98}, 3 (1989), 511--547.

\bibitem{drivas2023singularity}
{\sc Drivas, T.~D., and Elgindi, T.~M.}
\newblock Singularity formation in the incompressible {E}uler equation in
  finite and infinite time.
\newblock {\em EMS Surv. Math. Sci. 10}, 1 (2023), 1--100.

\bibitem{drivas2022anomalous}
{\sc Drivas, T.~D., Elgindi, T.~M., Iyer, G., and Jeong, I.-J.}
\newblock Anomalous dissipation in passive scalar transport.
\newblock {\em Arch. Ration. Mech. Anal. 243}, 3 (2022), 1151--1180.

\bibitem{Elgindi2024norm}
{\sc Elgindi, T.~M., and Liss, K.}
\newblock Norm growth, non-uniqueness, and anomalous dissipation in passive
  scalars.
\newblock {\em Arch. Ration. Mech. Anal. 248}, 6 (2024), Paper No. 120, 28.

\bibitem{elgindiUniversalMixersAll2019}
{\sc Elgindi, T.~M., and Zlato\v{s}, A.}
\newblock Universal mixers in all dimensions.
\newblock {\em Adv. Math. 356\/} (2019), 106807, 33.

\bibitem{Feng2019dissipation}
{\sc Feng, Y., and Iyer, G.}
\newblock Dissipation enhancement by mixing.
\newblock {\em Nonlinearity 32}, 5 (2019), 1810--1851.

\bibitem{Galdi2011introduction}
{\sc Galdi, G.~P.}
\newblock {\em An introduction to the mathematical theory of the
  {N}avier-{S}tokes equations}, second~ed.
\newblock Springer Monographs in Mathematics. Springer, New York, 2011.
\newblock Steady-state problems.

\bibitem{galeatiMixingGenericRough2023}
{\sc Galeati, L., and Gubinelli, M.}
\newblock Mixing for generic rough shear flows.
\newblock {\em SIAM J. Math. Anal. 55}, 6 (2023), 7240--7272.

\bibitem{garrett1979mixing}
{\sc Garrett, C.}
\newblock Mixing in the ocean interior.
\newblock {\em Dynamics of Atmospheres and Oceans 3}, 2-4 (1979), 239--265.

\bibitem{haynes2005transport}
{\sc Haynes, P.~H.}
\newblock Transport and mixing in the atmosphere.
\newblock In {\em Mechanics of the 21st Century: Proceedings of the 21st
  International Congress of Theoretical and Applied Mechanics, Warsaw, Poland,
  15--21 August 2004\/} (2005), Springer, pp.~139--152.

\bibitem{IyerKiselevXu2014}
{\sc Iyer, G., Kiselev, A., and Xu, X.}
\newblock Lower bounds on the mix norm of passive scalars advected by
  incompressible enstrophy-constrained flows.
\newblock {\em Nonlinearity 27}, 5 (2014), 973--985.

\bibitem{Kechris1995Classical}
{\sc Kechris, A.~S.}
\newblock {\em Classical descriptive set theory}, vol.~156 of {\em Graduate
  Texts in Mathematics}.
\newblock Springer-Verlag, New York, 1995.

\bibitem{Koochesfahani_Dimotakis_1986}
{\sc Koochesfahani, M.~M., and Dimotakis, P.~E.}
\newblock Mixing and chemical reactions in a turbulent liquid mixing layer.
\newblock {\em Journal of Fluid Mechanics 170\/} (1986), 83--112.

\bibitem{Lin2011optimal}
{\sc Lin, Z., Thiffeault, J.-L., and Doering, C.~R.}
\newblock Optimal stirring strategies for passive scalar mixing.
\newblock {\em J. Fluid Mech. 675\/} (2011), 465--476.

\bibitem{liuMixingEnhancementOptimal2008}
{\sc Liu, W.}
\newblock Mixing enhancement by optimal flow advection.
\newblock {\em SIAM J. Control Optim. 47}, 2 (2008), 624--638.

\bibitem{Lunasin2012Optimal}
{\sc Lunasin, E., Lin, Z., Novikov, A., Mazzucato, A., and Doering, C.~R.}
\newblock Optimal mixing and optimal stirring for fixed energy, fixed power, or
  fixed palenstrophy flows.
\newblock {\em J. Math. Phys. 53}, 11 (2012), 115611, 15.

\bibitem{mathewOptimalControlMixing2007}
{\sc Mathew, G., Mezi\'c, I., Grivopoulos, S., Vaidya, U., and Petzold, L.}
\newblock Optimal control of mixing in {S}tokes fluid flows.
\newblock {\em J. Fluid Mech. 580\/} (2007), 261--281.

\bibitem{mathew2005multiscale}
{\sc Mathew, G., Mezi\'c, I., and Petzold, L.}
\newblock A multiscale measure for mixing.
\newblock {\em Phys. D 211}, 1-2 (2005), 23--46.

\bibitem{mondena2018nonuniqueness}
{\sc Modena, S., and Sz\'ekelyhidi, Jr., L.}
\newblock Non-uniqueness for the transport equation with {S}obolev vector
  fields.
\newblock {\em Ann. PDE 4}, 2 (2018), Paper No. 18, 38.

\bibitem{nienow1997mixing}
{\sc Nienow, A.~W., Edwards, M.~F., and Harnby, N.}
\newblock {\em Mixing in the process industries}.
\newblock Butterworth-Heinemann, 1997.

\bibitem{ottinoKinematicsMixingStretching2004}
{\sc Ottino, J.~M.}
\newblock {\em The kinematics of mixing: stretching, chaos, and transport}.
\newblock Cambridge Texts in Applied Mathematics. Cambridge University Press,
  Cambridge, 1989.

\bibitem{seisMaximalMixingIncompressible2013}
{\sc Seis, C.}
\newblock Maximal mixing by incompressible fluid flows.
\newblock {\em Nonlinearity 26}, 12 (2013), 3279--3289.

\bibitem{Shnirelman1993lattice}
{\sc Shnirelman, A.~I.}
\newblock Lattice theory and flows of ideal incompressible fluid.
\newblock {\em Russian J. Math. Phys. 1}, 1 (1993), 105--114.

\bibitem{sverak2011course}
{\sc \v{S}ver\'{a}k, V.}
\newblock Course notes.
\newblock \url{http://math.umn.edu/~sverak/course-notes2011}, 2011/2012.

\bibitem{yaoMixingUnmixingIncompressible2017}
{\sc Yao, Y., and Zlato\v{s}, A.}
\newblock Mixing and un-mixing by incompressible flows.
\newblock {\em J. Eur. Math. Soc. (JEMS) 19}, 7 (2017), 1911--1948.

\bibitem{Young1942Generalized}
{\sc Young, L.~C.}
\newblock Generalized surfaces in the calculus of variations.
\newblock {\em Ann. of Math. (2) 43\/} (1942), 84--103.

\end{thebibliography}

\end{document}